\documentclass[USenglish]{article}	

\usepackage[utf8]{inputenc}				
\usepackage[big,online]{dgruyter}	
\usepackage{lmodern} 
\usepackage{microtype}
\usepackage{amsthm,amsmath,amsfonts,amssymb,commath}
\usepackage[all]{xy}
\usepackage[numbers,square,sort&compress]{natbib}

\theoremstyle{dgthm}
\newtheorem{theorem}{Theorem}
\newtheorem{corollary}{Corollary}
\newtheorem{proposition}{Proposition}
\newtheorem{lemma}{Lemma}

\newtheorem{problem}{Problem}

\theoremstyle{dgdef}
\newtheorem{definition}{Definition}
\newtheorem{example}{Example}

\begin{document}



\newcommand\xqed[1]{%
  \leavevmode\unskip\penalty9999 \hbox{}\nobreak\hfill
  \quad\hbox{#1}}
\newcommand\demo{\xqed{$\blacksquare$}}

\newcommand{\subjclass}[2][]{%
  \noindent{\textbf{Mathematics Subject Classification #1:}} #2\par
}

\title{Hearing Exotic Smooth Structures}
\runningtitle{Hearing Exotic Smooth Structures}

\author[1]{Leonardo F. Cavenaghi}
\author*[2]{João Marcos do Ó}
\author[3]{Llohann D. Speran\c ca} 
\runningauthor{L. F. Cavenaghi et al.}
\affil[1]{\protect\raggedright 
College of Arts and Sciences, Department of Mathematics, University of Miami, Ungar Bldg, 1365 Memorial Dr 515, Coral Gables, FL 33146\\ \&
Instituto de Matemática, Estatística e Computação Científica da Universidade Estadual de Campinas (Unicamp), Cidade Universitária, Campinas - SP, 13083-856, e-mail:leonardofcavenaghi@gmail.com}
\affil[2]{\protect\raggedright 
Departamento de Matemática -- Universidade Federal da Paraíba - Campus 1, 1º floor - Lot. Cidade Universitaria
58051-900, João Pessoa, PB, Brazil, e-mail: jmbo@mat.ufpb.br}
\affil[3]{\protect\raggedright 
Instituto de Ciência e Tecnologia -- Unifesp, Avenida Cesare Mansueto Giulio Lattes, 1201, 12247-014, São José dos Campos, SP, Brazil, e-mail: lsperanca@gmail.com}
	
	
\abstract{ This paper explores the existence and properties of \emph{basic} eigenvalues and eigenfunctions associated with the Riemannian Laplacian on closed, connected Riemannian manifolds featuring an effective isometric action by a compact Lie group. Our primary focus is on investigating the potential existence of homeomorphic yet not diffeomorphic smooth manifolds that can accommodate invariant metrics sharing common basic spectra. We establish the occurrence of such scenarios for specific homotopy spheres and connected sums. Moreover, the developed theory demonstrates that the ring of invariant admissible scalar curvature functions fails to recover the smooth structure in many examples. We show the existence of homotopy spheres with identical rings of invariant scalar curvature functions, irrespective of the underlying smooth structure.}

\maketitle


\section{Introduction}

A very classical and well-explored subject of geometric analysis is that of \emph{hearing shapes}. Following \cite{shapeofadrum}, it addresses the relationship between the geometry of Riemannian manifolds and the eigenvalues of its Laplacian operator acting on functions. Some remarkable developments are presented in \cite{milnor1964eigenvalues,milnor1964eigenvalueII,gordon1992one,gordon1996cant,buser1982spectrum, lauret2023spectrally,tanno1,tanno2,tanno3}.

Shortly, this paper mainly focuses on whether two homeomorphic but not diffeomorphic smooth manifolds admit invariant metrics sharing the same \emph{basic spectrum} (Definition \ref{def:singularspectrum}). With this aim, we use the analytical machinery and example constructions appearing in \cite{SperancaCavenaghiPublished} to compare the basic spectra of two homeomorphic but not diffeomorphic manifolds. Theorem \ref{thm:spectraintro} provides examples of pairs of homeomorphic but not diffeomorphic manifolds $M$ and $M'$ admitting invariant metrics $\mathsf{g}_M,~\mathsf{g}_{M'}$ that can be chosen with the same basic spectrum or not. Byproducts of the theory include showing that the ring of admissible scalar curvature functions do not distinguish smooth structures for plenty of homotopy spheres, Theorem \ref{ithm:aquiela}. This establishes the $G$-invariant version of the classical Kazdan--Warner problem \cite{kazdaninventiones} for some equivariantly related manifolds \cite{SperancaCavenaghiPublished}. 

The geometric constructions in this paper are based on the following setup. We consider a special case of principal $(G, H)$-bibundle in which $G=H$. Following \cite{Blohmann2008}, recall that a {\em groupoid} $G = (G_1 \rightrightarrows G_0)$ is a category in which each arrow in $G_1$ has an inverse.  Let $G$ and $H$ be Lie groupoids. A manifold $P$ with a left $G$-action and a right $H$-action which commute is called a smooth $G$-$H$-bibundle. We much benefit from the geometric constructions in \cite{SperancaCavenaghiPublished,speranca2016pulling}. Indeed, we consider smooth manifolds $P$ carrying commuting free $G$-actions whose orbit spaces $M, M'$ carry $G$-actions. We name this construction a $\star$-diagram; see Equation \eqref{eq:CDintro}. They constitute an example of \emph{principal $G$-bibundle} ($G=H$).

A $\star$-diagram $M\leftarrow P\rightarrow M'$ promotes a Hausdorff-Morita equivalence between the $G$-varieties $M$ and $M'$, as firstly remarked in \cite{SperancaCavenaghiPublished} and proved in \cite{cavenaghi2024newlookmilnor}. We hope our geometric realizations shed light on the study of Lie groupoid Morita equivalence via Hilsum–Skandalis maps \cite{Blohmann2008, Hilsum1987}. This specifies the study of leaf spaces \cite{Haefliger1984}, and in the particular case of Riemannian foliations, the isospectrality question for the basic Laplacian is rather natural \cite{Philippe1997}. For instance, it has been shown in \cite{Richardson2001} that the spectrum of the basic Laplacian of a Riemannian foliation can be identified with the $\mathrm{SO}(q)$-invariant spectrum of its Molino quotient manifold $W$ via the Molino frame bundle lift (which is itself a bibundle, see, e.g., \cite{Lin2020, Alexandrino2022}).

Let $(X,\mathsf{g})$ be a closed, connected $d$-dimensional Riemannian manifold with an effective isometric action by a compact Lie group $G$; $(X,\mathsf{g})$ is also called a \emph{Riemannian $G$-manifold}. Let $W^{k,p}(X), k\in\mathbb{Z}_{\geq 0}, p\in [1,\infty)$ denote the Sobolev space recovered by smooth functions on $X$ via taking the closure of $C^{\infty}(X)$ with respect to the norm:
\[\|u\|_{k,p}^p :=\sum\limits_{j=0}^k\int_X |\nabla^j u|^p \mathrm{d}\mu_{\mathsf{g}},\]
where $\nabla$ denotes the Levi--Civita connection of $\mathsf{g}$, $\nabla^ju$ is the $j^\text{th}$-covariant derivative of $u$, and the norm $|\nabla^j u|$ is computed using the tensor product metric induced by $\mathsf{g}$.

Once a Haar measure on $G$ is fixed, any element $u\in W^{k,p}(X)$ can be averaged along $G$ to produce a \emph{basic} or \emph{invariant} function $\bar u$, i.e., $\bar u(g x) = \bar u(x)$ for almost every $x\in X$ and every $g\in G$. We denote the set of such elements by $W^{k,p}_G(X)$. Being $C^{\infty}_G(X)$ the set of smooth basic functions in $(X,G)$, it can be shown (see e.g. \cite[Lemma 5.4, p. 91]{hebey1996sobolev}) that $\overline{C^{\infty}_G(X;\mathbb{R})}^{\|\cdot\|_{W^{k,p}(M)}} = W^{k,p}_G(X)$. Throughout this manuscript, we adopt the convention $W^{0,p}_G(X)=L^p_G(X).$ 

Let $-\Delta_{\mathsf{g}}:=-\mathrm{div}_{\mathsf{g}}(\nabla)$ be the (negative of the) Riemannian Laplacian of $\mathsf{g}$. It is well known that this is a formally self-adjoint, nonnegative elliptic operator, and in particular, its spectrum is a closed, countable, discrete, unbounded set consisting only of eigenvalues with finite multiplicities, which can therefore be arranged as a sequence of nonnegative real numbers $0=\lambda_0\leq \lambda_1\leq\lambda_2\leq\ldots\uparrow\infty$. This is known as the \emph{spectrum of $(X,\mathsf{g})$}. 

In the present context, where $(X,\mathsf{g})$ is a Riemannian $G$-manifold, we want to consider the subset of the spectrum consisting of eigenvalues that correspond to invariant eigenfunctions. Thus, we consider the $G$-invariant eigenvalue problem:
\begin{problem}\label{prob:invarianteigenvalue}
    Find $(u,\lambda)\in C^{\infty}_G(X)\times \mathbb  R$ such that
    \[-\Delta_{\mathsf{g}}u=\lambda u.\]
\end{problem}
To properly guarantee the existence of a solution to Problem \ref{prob:invarianteigenvalue}, we define:
\begin{definition}
Let $\mathcal{H}(X):=\{u\in W^{1,2}(X): \int_X u \mathrm{d}\mu_{\mathsf{g}} = 0\}$ and consider the subset $\mathcal{H}_G(X)$ consisting of invariant functions on $\mathcal{H}(X)$. A weak $G$-invariant solution for the Problem \ref{prob:invarianteigenvalue} is a pair $(u,\lambda)\in \mathcal{H}_G(X)\times \mathbb  R$ such that
    \begin{equation}\label{eq:weaksolution}
        \int_X\langle \nabla u,\nabla v\rangle\mathrm{d}\mu_{\mathsf{g}} = \lambda \int_Xuv\mathrm{d}\mu_{\mathsf{g}}
    \end{equation}
    for every $v\in \mathcal H(X)$. We say that $u$ is a \emph{basic weak eigenfunction of $-\Delta_{\mathsf{g}}$ with basic weak eigenvalue $\lambda$}.
\end{definition}

For a deep explanation of the basics related to the Laplacian spectra on manifolds, we refer the reader to \cite{chavel}. 

Note that the left-hand side in equation \eqref{eq:weaksolution} is given by
\[\langle u,v\rangle_{1,2} - \int_Xuv\mathrm{d}\mu_{\mathsf{g}}.\]
where the inner product $\langle u,v\rangle_{1,2}$ is the one associated with the norm $\|\cdot\|_{1,2}$. It can be checked, using that $X$ is closed and connected, that the norm $\|u\|^2:=\int_X |\nabla u|^2 \mathrm{d}\mu_{\mathsf{g}}$ defines a norm on $\mathcal{H}(X)$ which is equivalent to the norm $\|\cdot{}\|_{1,2}$. Associated with $\|\cdot\|$ we have the inner product
\[\langle u,v\rangle_{} := \int_X\langle\nabla u,\nabla v\rangle\mathrm{d}\mu_{\mathsf{g}}.\]
Hence, equation \eqref{eq:weaksolution} is equivalent to
\[\langle u,v\rangle_{}=\lambda \int_Xuv\mathrm{d}\mu_{\mathsf{g}},~\forall v\in \mathcal H(X).\]

On the other hand, for each fixed $u\in \mathcal H_G(X)$ the map
\[v\mapsto \int_Xuv\mathrm{d}\mu_{\mathsf{g}}\]
defines a linear functional with domain $(\mathcal H_G(X),\|\cdot\|)$. The Riesz representation theorem ensures the existence of $\lambda >0$ and $u^*_{\lambda}\in \mathcal H_G(X)$ such that
\[\langle u^*_{\lambda},v\rangle=\lambda \int_Xuv\mathrm{d}\mu_{\mathsf{g}}~\forall v \in \mathcal H_G(X).\]
The principle of symmetric criticality of Palais \cite{palais1979}
guarantees that the same holds for every  $v\in {\mathcal H}(X)$. Picking $v=u^*_{\lambda}$ shows that $u^*_{\lambda}=u$ a.e.. The classical regularity theory of PDEs guarantees the existence of a smooth invariant representative $u$ solving Problem \ref{prob:invarianteigenvalue}, i.e., $u\in C^{\infty}_G(X)$. 
\begin{definition}
    Let $(X,\mathsf{g})$ be a closed Riemannian manifold with isometric effective action by a compact connected Lie group $G$. Let $(u,\lambda)\in C^{\infty}_G(X)\times \mathbb  R_{>0}$ solving Problem \ref{prob:invarianteigenvalue}. We say that $\lambda$ is a \emph{basic eigenvalue} of $-\Delta_{\mathsf{g}}$ associated with the \emph{basic eigenfunction} $u$. 

    It can be also shown that the collection $\{\lambda\}$ of all possible invariant eigenvalues of $-\Delta_{\mathsf{g}}$ is increasing and unbounded, i.e., $0<\lambda_1\leq \lambda_2\leq \dots\uparrow \infty$. To such a collection we name \emph{the basic spectrum} of $-\Delta_{\mathsf{g}}$ and denote it as $\mathrm{Spec}_G(-\Delta_{\mathsf{g}})$.
\end{definition}

\begin{definition}\label{def:singularspectrum}
    Given two Riemannian manifolds $(X,\mathsf{g}),~(X',\mathsf{g}')$ with isometric actions by the same Lie group $G$, we say that they have the same \emph{basic spectrum} if $\mathrm{Spec}_G(-\Delta_{\mathsf{g}})=\mathrm{Spec}_G(-\Delta_{\mathsf{g}'})$.
\end{definition}

Definition \ref{def:singularspectrum} contrasts with other concepts already present in the literature. Given a Riemannian manifold $X$ with an isometric action by a Lie group $G$ (a Riemannian $G$-manifold), we know that $G$ has a natural representation $\tau_G$ on $L^2(X)$ where for each $f \in L^2(X)$ the function $\tau_G(f) := g.f$ is given by

\[
(g.f)(x) := f(g^{-1}x);
\]
note that, since $G$ acts via isometries, the representation $\tau_G$ commutes with the Laplacian $\Delta$ of $M$. Then, two Riemannian $G$-manifolds $X$ and $X'$ are said to be \emph{equivariantly isospectral} (with respect to $G$) if there is a unitary map $U : L^2(X) \to L^2(X')$ such that \cite{sutton2010equivariant,dryden2012advances}
\begin{itemize}
    \item[(i)] $U \circ \Delta = \Delta' \circ U$, that is, $X$ and $X'$ are isospectral,
    \item[(ii)] $U \circ \tau_G = \tau'_G \circ U$, i.e., the natural representations are equivalent via $U$.
\end{itemize}

Thus, the notion of isospectrality presented in Definition \ref{def:singularspectrum} is generally weaker than the above one of equivariant isospectrality, and it is, in fact, the one appearing in \cite{adelstein2017ginvariant}. In \cite{sunada1985riemannian}, Sunada studies whether two finite groups $H_1, H_2$ acting on a smooth manifold $X$ yield coinciding $H_i$-invariant spectrum on $X$. Then, in \cite{sutton2002isospectral}, Sunada's result is generalized to connected groups. Finally, Theorem 2.5 in \cite{adelstein2017ginvariant} goes beyond, stating the following: 
\begin{theorem}[Theorem 2.5 in \cite{adelstein2017ginvariant}]\label{thm:aldestein}
    Let $X$ be a compact Riemannian manifold and $G \leq \text{Isom}(X)$ a compact Lie group. Suppose that $H_1, H_2 \leq G$ are closed, representation-equivalent subgroups. Then, the $H_i$-invariant spectra of the Laplacian on $X$ are equal.
\end{theorem}
Our result in Proposition \ref{prop:joint}, combined with Example  \ref{ex:localmodelstar}, show that the conclusion in Theorem 2.5 in \cite{adelstein2017ginvariant} extends to every pair of manifolds $M, M'$ fitting a $\star$-diagram. 

Lastly, it may be the case that our appearing constructions bring insights on the following. Although two isospectral manifolds need not be isometric, further rigidity can be questioned. For instance, the result of Tanno in \cite{tanno1} ensures that \emph{any compact Riemannian manifold isospectral to a round sphere $(\mathrm{S}^n,\mathsf{g}_{\mathrm{round}})$ is necessarily isometric to $(\mathrm{S}^n,\mathsf{g}_{\mathrm{round}})$ if $n\leq 6$.} Up to the smooth Poincaré conjecture in dimension 4, \emph{exotic spheres} appear precisely firstly when $n=7$. It is still unknown whether round $7$ spheres are spectrally unique (in the sense of Tanno's result), but it is known that no exotic sphere can carry a round metric. Based on our results, we are tempted to think that \emph{the spectral-uniqueness can not occur for spheres in every dimension $n\geq 7$ where an exotic sphere exists}.

\section{Basic spectra of $G$-manifolds related by $\star$-diagrams}

We outline a general procedure for constructing exotic manifolds based on their classical counterparts, extensively discussed in \cite{speranca2016pulling, SperancaCavenaghiPublished,cavenaghi2019positive}. This is done by considering pairs of closed manifolds $M$, $M'$ carrying effective actions by a compact Lie group $G$, and fitting into a so-called \emph{$\star$-diagram} $M\leftarrow P\rightarrow M'$, which essentially realize $M$ and $M'$ as quotients of a single manifold $P$ by free commuting $G$-actions. Many exotic manifolds fit into these diagrams, and one can then use this structure to compare their invariant geometries once invariant metrics are considered. Notably, the concept of an \emph{exotic sphere} originated in the 1950s with J. Milnor's groundbreaking work \cite{mi}. Milnor introduced a family of $7$-dimensional manifolds $\Sigma^7$ homeomorphic to the classical sphere $\mathrm{S}^7$ but not diffeomorphic.

In the diagram \eqref{eq:CDintro} below, which will henceforth be called a $\star$-diagram, $P$ represents a principal $G$-manifold - a manifold equipped with a free action by a compact Lie group $G$, denoted by $\bullet$. This action implies $\pi$ defines a principal bundle over $M$ with total space $P$. We also assume the existence of another $G$-action, denoted by $\star$, which is both free and commutative with $\bullet$. This action makes $\pi'$ a principal bundle over $M'$ with total space $P$. We  encode everything in the following:

	\begin{equation}\label{eq:CDintro}
		\begin{xy}\xymatrix{& G\ar@{..}[d]^{\bullet} & \\ G\ar@{..}[r]^{\star} & P\ar[d]^{\pi}\ar[r]^{\pi'} &M'\\ &M&}\end{xy}
	\end{equation}	
Diagram \eqref{eq:CDintro} yields a principal $(G,G)$-bundle (shortly, a principal $G$-bundle), \cite{Blohmann2008}. We provide some explicit examples.
		\begin{example}[The Gromoll--Meyer exotic sphere]\label{ex:gromollmeyer}
		This construction first appeared in \cite{gm} and was first put in a $\star$-diagram in \cite{duran2001pointed} (see also \cite{speranca2016pulling}).
		Consider the compact Lie group
		\begin{equation}
			\mathrm{Sp}(2) = \left\{\begin{pmatrix} a & c \\b & d\end{pmatrix}\in \mathrm{S}^7\times \mathrm{S}^7~ \Big| ~a\bar{b} + c\bar{d} = 0\right\},\label{eq:Sp2}
		\end{equation} 
		where $a,b,c,d\in \mathbb  H$ are quaternions with their usual conjugation, multiplication, and norm.
		The projection $\pi:\mathrm{Sp}(2)\to \mathrm{S}^7$ of an element to its first row defines a principal $\mathrm{S}^3$-bundle with principal action:
		\begin{equation}\label{eq:GMprincipalaction}
			\begin{pmatrix} 
				a & c \\
				b & d 
			\end{pmatrix}\bar q = \begin{pmatrix}
				a & c\overline{q}\\
				b & d\overline{q}
			\end{pmatrix}.
		\end{equation}
		Gromoll--Meyer \cite{gm} introduced the $\star$-action
		\begin{equation}\label{eq:GMstaraction}
			q \begin{pmatrix} 
				a & c \\
				b & d 
			\end{pmatrix} = \begin{pmatrix} 
				qa\overline{q} & qc \\
				qb\overline{q} & qd 
			\end{pmatrix},
		\end{equation}
		whose quotient is an exotic 7-sphere. It all fits in the following diagram
		\begin{equation}\label{eq:CDGMintro}
			\begin{xy}\xymatrix{& \mathrm{S}^3\ar@{..}[d]^{\bullet} & \\ \mathrm{S}^3\ar@{..}[r]^{\star} & \mathrm{Sp}(2)\ar[d]^{\pi}\ar[r]^{\pi'} &\Sigma^7_{GM}\\ &\mathrm{S}^7&}\end{xy}
		\end{equation} \demo
  \end{example}

  As the next example shows, a $\star$-diagram such as \eqref{eq:CDintro} does not always produce a different manifold.
  \begin{example}
      [Pairs of diffeomorphic manifolds via $\star$-diagrams]\label{ex:localmodelstar}
        Let $M$ be a smooth manifold with an effective smooth action by a compact Lie group $G$, which we denote by $\cdot$. Consider the product manifold $M\times G$ with the following $\star$-action
        \[g\star (x,g'):= (g\cdot x,gg'),~x\in M,~g,g'\in G.\]

        Let $\bullet$ be the following $G$-action on $M\times G$:
        \[g\bullet (x,g'):= (x,(g')g^{-1}),~x\in M,~g,g'\in G.\]
        Both $\bullet,~\star$ are free and commuting actions on $M\times G$. Orbit maps for such actions are, respectively, $\pi:M\times G\rightarrow M,~(x,g')\mapsto x,~\pi':M\times G\rightarrow M,~(x,g')\mapsto (g')^{-1}x$. We can build the corresponding $\star$-diagram
        	\begin{equation}\label{eq:CDcheeger}
			\begin{xy}\xymatrix{& G\ar@{..}[d]^{\bullet} & \\ G\ar@{..}[r]^{\star} & M\times G\ar[d]^{\pi}\ar[r]^{\pi'} &M\\ &M&}\end{xy}
		\end{equation}
    \demo
  \end{example}
  
  A new construction is given next, encompassing some cohomogeneity-one manifolds.
 \begin{example}[Cohomogeneity-one manifolds]\label{ex:cohomogeneityone}
 Another class of examples are the manifolds constructed in Grove--Ziller \cite{gz}. Given integers $p_+,q_+,p_-,q_- \equiv 1 \pmod 4$, \cite{gz} produces a cohomogeneity-one manifold $(P^{10}_{p_-,q_-,p_+,q_+},(\mathrm{S}^3)^3)$. We further assume that $\gcd(p_-,q_-)=\gcd(p_+,q_+)=1$. Then one can observe that the subactions of $\mathrm{S}^3_{\bullet}=\mathrm{S}^3\times\{1\}\times\{1\}$ and of $\mathrm{S}^3_{\star}=\{1\}\times \Delta \mathrm{S}^3$ are commuting and free, where $\Delta \mathrm{S}^3$ is the diagonal in $\mathrm{S}^3\times \mathrm{S}^3$. They fit in the diagram
		\begin{equation}\label{eq:CDGZ}
			\begin{xy}\xymatrix{& \mathrm{S}^3\ar@{..}[d]^{\bullet} & \\ \mathrm{S}^3\ar@{..}[r]^{\hspace{-18pt}\star} & P^{10}_{p_-,q_-,p_+,q_+}\ar[d]^{\pi}\ar[r]^{\pi'} &M'_{p_-,q_-,p_+,q_+}\\ &M_{p_-,q_-,p_+,q_+}&}\end{xy}
		\end{equation}
		Here, $M'_{p_-,q_-,p_+,q_+}$ is the $\mathrm{S}^3$-bundle over $\mathrm{S}^4$ classified by the transition function $\alpha:\mathrm{S}^3\to \mathrm{SO}(4)$ defined by $\alpha(x)v=x^kvx^l$, where $k=(p^2_--p^2_+)/8$ and $l=-(q^2_--q^2_+)/8$. \demo
	\end{example}
	
  As observed in \cite{SperancaCavenaghiPublished}, the actions $\bullet$ and $\star$ commute so that $\star$ descends to a non-trivial action on $M$, as well as $\bullet$ descends to a non-trivial action $M'$. Moreover, it is possible to regard $M$ and $M'$ with $G$-invariant Riemannian metrics $\mathsf{g}_M$ and $\mathsf{g}_{M'}$, respectively, such that the metric spaces $M/G$ and $M'/G$ are isometric. 
  
    \begin{lemma}[Corollary 5.2 in \cite{SperancaCavenaghiPublished}]\label{lem:totallygeodesicmetric}
     Let $M\leftarrow P\rightarrow M'$ shortly denote a $\star$-diagram such as \eqref{eq:CDintro} with structure group $G$. There exists a $G\times G$-invariant metric $\mathsf{g}_{\omega}$ on $P$ that induces $G$-invariant Riemannian metrics $\mathsf{g}_{M},~\mathsf{g}_{M'}$ on $M$, and $M'$, respectively, such that the metric spaces $M/G,~M'/G$ are isometric. Moreover, the $\bullet, \star$-action fibers on $P$ are totally geodesic.
 \end{lemma}
  \begin{proof}[Sketch of the Proof]
  Let $\omega : TP \to \mathfrak g$ be a connection $1$-form associated with $\pi:P\rightarrow M$. Proposition 5.1 in \cite{SperancaCavenaghiPublished} teaches us we can assume that for any $r\in G$ it holds $\omega_{rp}(rX)=\omega_p(X)$ for any $p\in P,~X\in T_pP$. Let $Q$ be a bi-invariant metric on $G$ and $\mathsf{g}_M$ be any $G$-invariant Riemannian metric on $M$. Regard $P$ with the  Kaluza--Klein metric $\mathsf{g}_{\omega} = \pi^*\mathsf{g}_M + Q\circ \omega\otimes\omega$. Since $\mathsf{g}_{\omega}$ is $G\times G$-invariant it yields a $G$-invariant Riemannian metric $\mathsf{g}_{M'}$ on $M'$.
  
  One straightforwardly checks that
	\[\mathcal H'' := (T(G\times G)p)^{\perp_{\mathsf{g}_{\omega}}} \cong (TG\pi(p))^{\perp_{\mathsf{g}_M}} \cong (TG\pi'(p))^{\perp_{\mathsf{g}_{M'}}},~\forall p\in P.\]
 Then, any geodesic orthogonal to an orbit of the $\star$-action on $M$ can be mapped (through horizontal lifting from $M$ and $\pi'$-projection) to a geodesic orthogonal to an orbit of the $\bullet$-action on $M'$ with the same length. The orbits are totally geodesic on $P$ (for both actions) because Kaluza--Klein metrics are connection metrics.
\end{proof}

Any smooth $G\times G$-invariant function $\phi : P \to \mathbb{R}$ gives rise to smooth and $G$-invariant functions on both $M$ and $M'$ via composing $\phi$ in the right with $\pi,~\pi'$, respectively. In fact, any smooth invariant functions on $M,~M'$ arise in this fashion. Let $\mathsf{g}_P$ denote an arbitrary $G\times G$-invariant metric on $P$, and let $\mathsf{g}_M,$ and $\mathsf{g}_{M'}$ be invariant Riemannian metrics on $M,~M'$, respectively, such that the projections $\pi,~\pi'$ define Riemannian submersions. Then, the \emph{horizontal Laplacians} on  $M$ and $M'$ are related to the Laplacian in $P$ as follows, see \cite[Section 2.1.4, p.53]{gw}:
	\begin{align}\label{eq:system}
	 -\Delta_{\mathsf{g}_P}\phi &= -\Delta_{\mathsf{g}_M}\phi + \mathrm{d}\phi(H^{\pi}),\\
	 -\Delta_{\mathsf{g}_P}\phi &= -\Delta_{\mathsf{g}_{M'}}\phi + \mathrm{d}\phi(H^{\pi'}),
	\end{align}
	where $H^{\pi}$ stands for the mean curvature vector of the fibers according to the $\bullet$-action on $P$ and $H^{\pi'}$ to the mean curvature vector of the fibers according to $\star$. Since the ring isomorphisms $C^{\infty}_{G\times G}(P)\cong C^{\infty}_G(M)\cong C^{\infty}_G(M')$ hold, one can identify the Sobolev spaces
 \[W^{1,2}_{G\times G}(P)\cong W^{1,2}_G(M)\cong W^{1,2}_G(M').\]
 Consequently,
 \[\mathcal H_{G\times G}(P)\cong \mathcal H_{G}(M)\cong \mathcal H_G(M').\]
	\begin{problem}\label{prob:thetwo}
	Is there $\phi \in C^{\infty}_{G\times G}(P)$ simultaneously a smooth solution to the invariant eigenvalue problems \[\Delta_{\mathsf{g}_{M}}\phi = -\lambda \phi, ~\Delta_{\mathsf{g}_{M'}}\phi = -\lambda' \phi\]
 for some $\lambda, \lambda'\in \mathbb  R$? If so, is it true that $\lambda=\lambda'$? Moreover, do the basic spectra $\mathrm{Spec}(-\Delta_{\mathsf{g}_M})$ and $\mathrm{Spec}(-\Delta_{\mathsf{g}_{M'}})$ coincide?
	\end{problem}

Related to Problem \ref{prob:thetwo} we prove:
\begin{proposition}\label{prop:joint}
   Let $(M,\mathsf{g}_M)\leftarrow (P,\mathsf{g}_P)\rightarrow (M',\mathsf{g}_{M'})$ shortly denote a $\star$-diagram where $\mathsf{g}_{M},~\mathsf{g}_P,~\mathsf{g}_{M'}$ are as in Lemma \ref{lem:totallygeodesicmetric}. Let $\Phi$ be the collection of $G$-invariant eigenfunctions of $-\Delta_{\mathsf{g}_M}$ in $(M,\mathsf{g}_M)$. Then for each $\phi \in \Phi$ we have that $-\Delta_{\mathsf{g}_{M'}}\phi=\lambda' \phi$ with $\lambda'=\lambda$ where $-\Delta_{\mathsf{g}_{M}}\phi=\lambda \phi$. Moreover, the roles of $(M,\mathsf{g}_M)$ and $(M',\mathsf{g}_{M'})$ can be interchanged, in the sense that we can start with $\Phi$ as a set of invariant eigenfunctions in $(M, \mathsf{g}_{M'})$ for $-\Delta_{\mathsf{g}_{M'}}$.
\end{proposition}
\begin{proof}
    Let $\Phi$ be a set of invariant eigenfunctions for $-\Delta_{\mathsf{g}_M}$. From the compactness of $M$ and standard elliptic theory, $\Phi$ constitutes an orthonormal (Hilbert) basis to $\mathcal H_G(M)$ and hence consists in a Schauder basis to $\mathcal H_{G\times G}(P)$ \cite[p.16]{chavel}. In this manner, $\Phi$ descends to a Schauder basis to ${\mathcal H}_G(M')$. We now pose the following variant of Problem \ref{prob:invarianteigenvalue}. To find $\phi \in \Phi$ and $\lambda' \in \mathbb  R$ such that
    \[\lambda'\int_{M'}\phi v =\int_{M'}\mathsf{g}'(\nabla'\phi,\nabla v),~\forall v\in {\mathcal H}_G(M').\]
    To solve this problem, we consider the functional $J_{\lambda'}(u):=\int_{M'}|\nabla'u|^2-\lambda'\int_{M'}u^2$ with domain in ${\mathcal H}_G(M')$ and show that for each $n\in \mathbb  N$, picking $\lambda'_n=\lambda_n\in \mathbb  R$ for $\lambda_n\in \mathrm{Spec}_G(-\Delta_{\mathsf{g}})$, then $\left\{\phi_j^n\right\}_{j=1}^{\mathrm{dim} \ker(-\Delta_{\mathsf{g}}-\lambda_n)}:=\Phi\cap \ker(-\Delta_{\mathsf{g}}-\lambda_n)$ satisfies $\mathrm{d}J_{\lambda'_n}(\phi_j^n)(v)=0~\forall j\in \{1,\ldots,\mathrm{dim} \ker(-\Delta_{\mathsf{g}}-\lambda_n)\},~\forall v \in {\mathcal H}_G(M')$. Using that $\Phi$ is a Schauder basis to ${\mathcal H}_G(M')$ it suffices to show for each $n\in \mathbb  N$ we have
    \[\int_{M'}\left[\mathsf{g}_{M'}(\nabla'\phi_j^n,\nabla'\phi_k)-\lambda_n\phi_j^n\phi_k\right]=0
    ,~\forall k\in \mathbb  N,~\forall j.\]

   Note that for each $p \in P$ we have for every $j, n, k$ that
    \begin{align*}&\nabla'\phi_k(\pi'(p))\perp T_{\pi'(p)}G\pi'(p),~\nabla'\phi_j^n(\pi'(p))\perp T_{\pi'(p)}G\pi'(p),\\&\nabla\phi_k(\pi(p))\perp T_{\pi(p)}G\pi(p),~\nabla\phi_j^n(\pi(p))\perp T_{\pi(p)}G\pi(p).\end{align*} As appearing in the proof of Lemma \ref{lem:totallygeodesicmetric}, the orthogonal spaces to each $G$-orbit in $(M,\mathsf{g}_{M})$ and $(M',\mathsf{g}_{M'})$ are isometric. Therefore, $\mathsf{g}_{M'}(\nabla'\phi_j^n,\nabla'\phi_k)=\mathsf{g}_M(\nabla\phi_j^n,\nabla\phi_k)$. Finally, using that $M/G$ and $M'/G$ are isometric, we have
    \begin{align*}
        \int_{M'/G}\left[\mathsf{g}_{M'}(\nabla'\phi_j^n,\nabla'\phi_k)-\lambda_n\phi_j^n\phi_k\right]&=\int_{M/G}\left[\mathsf{g}_{M}(\nabla \phi_j^n,\nabla\phi_k)-\lambda_n\phi_j^n\phi_k\right].
    \end{align*}
    Since $\Phi$ collects the solutions of Problem \ref{prob:invarianteigenvalue}, the version of the Fubini theorem appearing in \cite[Satz 1, p.210]{fubinirefenrece} can be applied to conclude the desired result. \qedhere
\end{proof}

\begin{definition}
    Let $M\leftarrow P\rightarrow M'$ shortly denote a $\star$-diagram. To a set $\Phi$ solving Problem \ref{prob:invarianteigenvalue} in $(M,\mathsf{g}_M)$ and $(M',\mathsf{g}_{M'})$ simultaneously we name a \emph{joint invariant eigenfunction set}.
\end{definition}

Our next result shows that when existing a joint invariant eigenfunction set $\Phi$ necessarily $\mathsf{g}_{M}, ~\mathsf{g}_{M'}$ are isospectral. 

\begin{proposition}
    Let $M\leftarrow P\rightarrow M'$ shortly denote a $\star$-diagram such as \eqref{eq:CDintro} with structure group $G$ where $M, P$ and $M'$ are closed and connected. Let $\mathsf{g}_M,~\mathsf{g}_{M'}$ be $G$-invariant metrics on $M$, and $M'$, respectively, induced by a $G\times G$-invariant metric $\mathsf{g}_P$ on $P$. Assume an existing joint invariant eigenfunction set $\Phi$. Then the $G$-invariant spectra of $-\Delta_{\mathsf{g}_M}$ and $-\Delta_{\mathsf{g}_{M'}}$ coincide.
\end{proposition}
\begin{proof}
Observe that the function $p\mapsto\frac{\mathrm{vol}(G^{\star}p)}{\mathrm{vol}(G^{\bullet}p)}$ is constant on $P$. By connectedness, this boils down to showing that the functions $p\mapsto \mathrm{vol}(G^{\star}p)$ and $p\mapsto \mathrm{vol}(G^{\bullet}p)$ are locally constant on $P$. Now, $P$ is a principal $G$-bundle for both $G$-actions $\star$ and $\bullet$, so one can relate sufficiently close fibers of the star action, as well as sufficiently close fibers of the dot action, using principal $G$-bundle trivializations. Combining this with the fact that the metric on $P$ is $G\times G$-invariant, and using Fubini's theorem, one can show that the volume of these sufficiently close fibers of $P$ is indeed the same, say equal to $V'$ for the $\star$-action and equal to $V$ for the $\bullet$ action.

On the other hand, Equation \eqref{eq:system} ensures that any $\phi\in C^{\infty}_{G\times G}(P)$ which descends to solutions to the eigenvalue problems on $M$ and $M'$, respectively, can be characterized by
	\[ -\lambda \phi + \mathrm{d}\phi(H^{\pi}-H^{\pi'}) = -\lambda' \phi.\]
	Thus, $\lambda' = \lambda$ holds if, and only if, $H^{\pi}-H^{\pi'}\in \ker \mathrm{d}\phi$. Fix $p\in P$. The mean curvature vector along the $G$-orbit through $p$ is given by $-H = \nabla \log \mathrm{vol}(Gp)$ -- \cite[Lemma 5.2]{alexandrino2016mean}. Hence,
	\[H^{\pi} - H^{\pi'} = \nabla \log \left(\frac{\mathrm{vol}(G^{\star}p)}{\mathrm{vol}(G^{\bullet}p)}\right),\]
	so
	\[\mathrm{d}\phi(H^{\pi} - H^{\pi'}) = \langle \nabla \phi, \nabla \log \left(\frac{\mathrm{vol}(G^{\star}p)}{\mathrm{vol}(G^{\bullet}p)}\right)\rangle=\langle \nabla \phi, \nabla \log \dfrac{V'}{V}\rangle=0.\]
 Therefore,
\[-\lambda \phi = -\lambda' \phi.\qedhere\]
 \end{proof}
	
  An appearing question is whether we can produce invariant metrics on $M$ and $M'$, which are invariant and not isospectral. We obtained the following, to be proved in Section \ref{sec:basicspectra}.
\begin{theorem}\label{thm:isospectralornot}
  For any $\star$-diagram $M\leftarrow P\rightarrow M'$ with compact connected structure group and  $M, P$ and $M'$ being closed and connected, there exists invariant metrics $\mathsf{g}$ on $M$ and $\mathsf{g}'$ on $M'$ so that $(M,\mathsf{g}),~(M',\mathsf{g}')$ have different basic spectrum.
\end{theorem}
The combination of Proposition \ref{prop:joint} with Theorem \ref{thm:isospectralornot} and some explicit realization of exotic manifolds equivariantly related to their classical counterpart allow us to show
 	\begin{theorem}\label{thm:spectraintro}
The following pair of homeomorphic but not diffeomorphic manifolds admit Riemannian metrics invariant by the same group of isometries that can be chosen admitting the same basic spectrum or not:
    \begin{enumerate}
        \item $\mathrm{S}^7,~\#^k\Sigma_{GM}^7$ where $\Sigma_{GM}^7$ is the Gromoll--Meyer exotic sphere (Example \ref{ex:gromollmeyer}), $k\in\mathbb  N$ and $\#$ denote the connected sum,
        \item $\mathrm{S}^8,~\Sigma^8$ where $\Sigma^8$ is the only $8$-dimension exotic sphere,
        \item $\mathrm{S}^{10},~\Sigma^{10}$ where $\Sigma^{10}$ is a generator of the index two homotopy subgroup of $10$-dimension homotopy spheres that bound spin manifolds,
        \item $\mathrm{S}^{4n+1},~\Sigma^{4n+1}$ where $\Sigma^{4n+1}$ are the known \emph{Kervaire spheres},
        \item the pair of total spaces of the bundles appearing in Example \ref{ex:cohomogeneityone},
        \item the classical and exotic realization of the manifolds appearing in Example \ref{ex:moreconnectedsums}.
    \end{enumerate}
	\end{theorem}

With Theorem \ref{thm:spectraintro} in hands, it is natural to ask whether invariant geometric objects \textit{ignore} the chosen smooth structure to a fixed underlined topological space. We prove
\begin{theorem}[$G$-invariant Kazdan--Warner problem on $\star$-diagrams]\label{ithm:aquiela}
	Consider a $\star$-diagram $M\leftarrow P\to M'$ with $G$ connected. Then a basic function on $M$ is the scalar curvature of a $G$-invariant metric on $M$ if and only if it is the scalar curvature of a $G$-invariant metric on $M'$ if and only if it lifts to the scalar curvature of a $G{\times}G$-invariant metric on $P$.
\end{theorem}
\begin{corollary}\label{cor:kW}
    The following pair of homeomorphic but not diffeomorphic manifolds admit the same ring of invariant scalar curvature functions for a certain isometry group:
    \begin{enumerate}
        \item $\mathrm{S}^7,~\#^k\Sigma_{GM}^7$ where $\Sigma_{GM}^7$ is the Gromoll--Meyer exotic sphere (Example \ref{ex:gromollmeyer}), $k\in\mathbb  N$ and $\#$ denote the connected sum,
        \item $\mathrm{S}^8,~\Sigma^8$ where $\Sigma^8$ is the only $8$-dimension exotic sphere,
        \item $\mathrm{S}^{10},~\Sigma^{10}$ where $\Sigma^{10}$ is a generator of the index two homotopy subgroup of $10$-dimension homotopy spheres that bound spin manifolds,
        \item $\mathrm{S}^{4n+1},~\Sigma^{4n+1}$ where $\Sigma^{4n+1}$ are the known \emph{Kervaire spheres},
        \item  the pair of total spaces of the bundles appearing in Example \ref{ex:cohomogeneityone},
             \item the classical and exotic realization of the manifolds appearing in Example \ref{ex:moreconnectedsums}.
    \end{enumerate}
\end{corollary}

In Section \ref{sec:equivariant}, we both prove Theorem \ref{ithm:aquiela} and furnish the examples appearing in Theorem \ref{thm:spectraintro} and Corollary \ref{cor:kW}.

\section{On the realizability of scalar curvature functions on homotopy spheres}
\label{sec:equivariant}

Since Milnor introduced the first examples of \textit{exotic manifolds} \cite{mi}, many new exotic spaces have been produced. For instance, there are uncountable many pairwise non-diffeomorphic structures on $\mathbb{R}^4$ (see \cite{taubes}); exotic manifolds not bounding spin manifolds \cite{hitchin1974harmonic}; exotic projective spaces, and connected sums of exotic manifolds \cite{SperancaCavenaghiPublished}. Following \cite{duran2001pointed,speranca2016pulling,Cavenaghi2023,cavenaghi2019positive}, several realizations of exotic manifolds are obtained using $\star$-diagrams such as \eqref{eq:CDintro}. This section explores the relationship between the admissibility of invariant scalar curvature functions on the manifolds $M$ and $M'$, culminating in the proof of  Theorem \ref{ithm:aquiela}. 

The first construction of exotic manifolds uses the classical \textit{Reeb's Theorem} to show that specific $7$-dimensional total spaces of sphere bundles are homeomorphic to a standard sphere. Moreover, one can recover the smooth structure of a manifold through its space of smooth functions (see, for example, \cite[Problem 1-C]{milnor1974characteristic}). However, in the presence of a $\star$-diagram, the set of basic functions of $M$ and $M'$ are naturally identified since they are naturally identified with the space of $G\times G$-invariant functions on $P$, proving that the set of basic functions does not recover $(M, G)$. Theorem \ref{ithm:aquiela} reinforces this fact in the sense that invariant scalar curvature functions should not distinguish smooth structures.

\begin{proof}[Proof of Theorem \ref{ithm:aquiela}]
Let $M\leftarrow P \rightarrow M'$ be a $\star$-diagram and $p\in P$.
Denote $\pi(p)=x$ and $\pi'(p)=x'$.
First, a straightforward calculation shows that the isotropy groups $G_x$ and $G_{x'}$ are isomorphic -- \cite[Theorem 2.2 and Proposition 5.3]{SperancaCavenaghiPublished}. Therefore, if $G$ acts effectively on $M$, it does on $M'$. We now observe that if $G$ has a non-Abelian Lie algebra, then both $(M, G), (M', G)$ admit any invariant function as scalar curvature of some invariant Riemannian metric. This holds since a non-Abelian Lie algebra ensures the existence of an invariant metric of positive scalar curvature -- \cite{lawson-yau, Cavenaghi2023}. The main result in \cite{kazdaninventiones} ensures that any scalar curvature function is admissible as scalar curvature of some metric close to this positively curved metric. The arguments in \cite{cavenaghi2023kazdanwarner} show that the resulting metric is $G$-invariant.

Finally, if $G$ is Abelian, then it is a torus. Therefore, we can apply \cite[Theorem 2.2]{wiemeler2016circle} to conclude that $P$ admits a $G\times G$-invariant metric with positive scalar curvature if and only if both $M$ and $M'$ admit $G$-invariant metrics with positive scalar curvature, as wanted. \qedhere
\end{proof}

The remaining section presents many examples where Theorems \ref{thm:isospectralornot}-\ref{ithm:aquiela} can be applied. We explicitly build the examples in the statements of Theorem \ref{thm:spectraintro} and Corollary \ref{cor:kW}.

\begin{example}[Pulling-back $\star$-diagrams]
 Consider a $\star$-diagram $M\leftarrow P\rightarrow M'$ and a $G$-manifold $N$. Let $\phi: N\to M$ be a $G$-equivariant function. We can pull-back this $\star$-diagram producing a new quotient $(\phi^*P)/G=N'$. The pull-back construction was applied in \cite{speranca2016pulling,SperancaCavenaghiPublished} to obtain the following examples:
\begin{itemize}
\item[$(\Sigma^7_k)$:] consider $\phi:\mathrm{S}^7\to \mathrm{S}^7$ as the octonionic $k$th fold power. Then the corresponding $\star$-diagram $\mathrm{S}^7\leftarrow \phi^*\mathrm{Sp}(2)\to (\mathrm{S}^7)'$ yields $(\mathrm{S}^7)'$ diffeomorphic to the connected sum of $k$ times $\Sigma^7_{GM}$ ; 
\item[($\Sigma^8)$:] there is a $\mathrm{S}^3$-equivariant suspension $\eta:\mathrm{S}^8\to \mathrm{S}^7$ of the Hopf map $\mathrm{S}^3\to \mathrm{S}^2$ whose quotient $(\mathrm{S}^8)'=\eta^*\mathrm{Sp}(2)/\mathrm{S}^3$ is the only exotic 8-sphere;
\item[$(\Sigma^{10})$:] there is a $\mathrm{S}^3$-equivariant suspension $\theta: \mathrm{S}^{10}\to \mathrm{S}^7$ of a generator of $\pi_6\mathrm{S}^3$ whose induced $\star$-quotient $(\mathrm{S}^{ 10})'$ is a generator of the index two subgroup os homotopy 10-spheres that bound spin manifolds;
\item[$(\Sigma^{4n{+}1})$:] the frame bundle $\rm{pr}_n:\mathrm{SO}(2n{+}2)\to \mathrm{S}^{2n+1}$ can be also seen as a $\star$-diagram: one can endow $\mathrm{SO}(2n{+}2)$ with both the right and left multiplication by $\mathrm{SO}(2n{+}1)$. In this case, $M=M'=\mathrm{S}^{2n+1}$. However, there is a pull-back map $J\tau:\mathrm{S}^{4n{+}1}\to \mathrm{S}^{2n+1}$,
whose $\star$-diagram $\mathrm{S}^{4n{+}1}\leftarrow (J\tau)^*\mathrm{SO}(2n{+}2)\to (\mathrm{S}^{4n{+}1})'$ has $(\mathrm{S}^{4n{+}1})'$ diffeomorphic to a Kervaire sphere. Moreover, one can `reduce' $G=\mathrm{SO}(2n{+}1)$ in to either $\mathrm{U}(n)$ or $\mathrm{Sp}(n)$ (supposing $n$ odd for the last).
\end{itemize}
\demo
\end{example}

\begin{example}[Gluing and connected sums]
Consider $W_1, W_2$ manifolds with boundaries and equips with $G$-actions. Assume that $f:\partial W_1\to\partial W_2$ is an equivariant diffeomorphism. Then one can produce a new manifold $W=W_1\cup_f W_2$ by gluing $W_1, W_2$ via $f$. $W$ thus inherits a natural smooth $G$-action whose restrictions to $W_1,W_2\subset W$ coincide with the original actions.

Interesting examples arise in the following way: let $(M_1, G),(M_2, G)$ be closed manifolds with $G$-actions. Suppose that $x_i\in M_i$, $i=1,2$ have the same \textit{orbit type}, that is, $G_{x_1}, G_{x_2}$ are subgroups in the same conjugacy class and their isotropy representations are equivalent. Then, one can remove small tubular neighborhoods of the orbits $Gx_1, Gx_2$ and glue the boundaries together. In particular, if $x_1,x_2$ are fixed points with equivalent isotropy representations, one can perform a connected sum. 

More generally, one can consider the case where $(M,G)$ admits an equivariant embedding of $(\mathrm{S}^k\times D^{l+1},G)$, where $G$ acts on $\mathrm{S}^k\times D^{l+1}$ is equipped with a linear action. In this case, one can perform surgery along $\psi$. At this point, $\star$-diagrams become quite useful since corresponding orbits in $M, M'$ often have the same orbit type, as the next lemma points out.

\begin{lemma}\label{lem:isotropy}
Let $M\leftarrow P \rightarrow M'$ be a $\star$-diagram and $p\in P$. There is a group isomorphism $\phi:G_{\pi(p)}\to G_{\pi'(p)}$  and a linear isomorphism $\psi$ such that $\rho_{\pi(p)}=\psi\rho_{\pi'(p)}\phi$. 
\end{lemma}
\begin{proof}

Denote $\pi(p)=x$ and $\pi'(p)=x'$. For simplicity, we only prove the last assertion since the existence of $\phi$ follows by direct computation.

Note that $\mathrm{d}\pi_p,\mathrm{d}\pi'$ define isomorphisms between normal spaces:
\[\frac{T_xM}{T_xGx} \stackrel{\overline{\mathrm{d}\pi}}{\longleftarrow}\frac{T_pP}{T_p(G{\times}G)p} \stackrel{\overline{\mathrm{d}\pi'}}{\longrightarrow} \frac{T_{x'}M'}{T_{x'}Gx'}. \]
Moreover, since $\pi,\pi'$ commute with the (respective complementary) actions, 
\[\rho_x(h)\overline{\mathrm{d}\pi}(\overline{v}) ~{\rightarrow}~\rho_p(h,\phi(h))\overline{v} ~{\leftarrow}~ \rho_{x'}(\phi(h))\overline{\mathrm{d}\pi'}(\overline{v}),\]
inducing the desired identification.
\end{proof}

The isomorphism $\phi$ in the examples $(\Sigma^7_k)$-$(\Sigma^{4n+1})$ are all of the form $\phi(h)=ghg^{-1}$ for $g\in G$ only depending on $p$. In such cases, every pair of points $x,x'$, $\pi^{-1}(x)\cap (\pi')^{-1}(x')\neq \emptyset$, have the same orbit type. 

We claim that several surgeries can be done equivariantly on the manifolds resulting in these examples. Moreover, such surgeries can be done by keeping the $\star$-diagram apparatus. We give more details on the $(\Sigma_k^7)$-case to illustrate the assertion.

In this case, $\mathrm{S}^3$ acts on $\mathrm{S}^7$ as $q(a,b)^T=(qa\bar q,qb\bar q)^T$. This action is inherited from the representation  $\tilde \rho:\mathrm{S}^3\to \mathrm{SO}(8)$ defined by $2\rho_0\oplus 2\rho_1$, where $\rho_0$ and $\rho_1$ stands for the trivial representation and the representation defined by the composition of the double-cover $\mathrm{S}^3\to \mathrm{SO}(3)$ and the standard action of $\mathrm{SO}(3)$ in $\mathbb  R^3$. I.e., $\tilde \rho$ is the double suspension of the bi-axial action of $\mathrm{SO}(3)$ in $\mathbb  R^6$, up to a double-cover.

Note that $(a,b)^T$ is a fixed point of $\tilde \rho$ whenever $a,b\in\mathbb  R$ and consider another manifold $(M^7,\mathrm{S}^3)$ with a fixed point $p$ whose isotropy representation is $\rho_0\oplus 2\rho_1$. One can produce a standard degree-one equivariant map $\phi: M^7\to \mathrm{S}^7$ by `wrapping' $\mathrm{S}^7$ with an open ball centered at the fixed point and sending the remaining of $M$ to the antipodal of $\phi(p)$. As in \cite[Theorem 4.1]{SperancaCavenaghiPublished}, the induced $\star$-diagram results in $M\leftarrow \phi^*P\to M\# \Sigma^7_k$.

To proceed with the surgery process, note that $(\mathrm{S}^7,\mathrm{S}^3)$ admits the equivariant submanifolds below. We omit the explicit embeddings and use the representation instead of $G$ as the notation $(M, G)$ to present more detailed information.
\begin{align*}
(\mathrm{S}^1,2\rho_0)\times (D^6,2\rho_1)&=\{(a,b)^T\in \mathrm{S}^7 ~|~(\rm{Re}(a),\rm{Re}(b))\neq (0,0)\} ;\\      (\mathrm{S}^2,\rho_1)\times (D^5,2\rho_0\oplus \rho_1)&=\{(a,b)^T\in \mathrm{S}^7 ~|~\rm{Im}(a)\neq 0\} ;\\
(\mathrm{S}^3,\rho_0\oplus\rho_1)\times (D^4,\rho_0\oplus \rho_1)&=\{(a,b)^T\in \mathrm{S}^7 ~|~a\neq 0\}; \\
(\mathrm{S}^4,2\rho_0\oplus\rho_1)\times (D^3,\rho_1)&=\{(a,b)^T\in \mathrm{S}^7 ~|~(a,\rm{Re}(b))\neq (0,0)\} ;\\
(\mathrm{S}^5,2\rho_1)\times (D^2,2\rho_0)&=\{(a,b)^T\in \mathrm{S}^7 ~|~(\rm{Im}(a),\rm{Im}(b))\neq (0,0)\} ;\\
(\mathrm{S}^6,\rho_0\oplus 2\rho_1)\times (D^1,\rho_0)&=\{(a,b)^T\in \mathrm{S}^7 ~|~(\rm{Im}(a),b)\neq (0,0)\}.
\end{align*}

Except for $\mathrm{S}^1\times D^6$ and $\mathrm{S}^4\times D^3$, every submanifold above can lie in an arbitrarily small region of a fixed $\rm{Re}(a)$. In particular, arbitrarily, many of these surgeries can be performed. 

Moreover, we conclude that such surgeries can be performed by preserving infinitely many fixed points. Therefore, the above degree-one map can be considered, producing a $\star$-diagram over the new manifold $M$. Although the connected sum applied to this context seems \textit{ad-hoc}, the resulting manifold $(\phi^*P)/G=M\# \Sigma$ is the same manifold one obtains by performing the same surgeries on $\Sigma$. \demo
\end{example}
\begin{example}[More connected sums]\label{ex:moreconnectedsums}
A list of manifolds whose fixed points have isotropy representations isomorphic to the ones of the examples $(\Sigma^7_k)$-$(\Sigma^{4n+1})$ are found in \cite{SperancaCavenaghiPublished}. We compile it here:
\begin{proposition}[Cavenaghi--Sperança]\label{thm:llohann}
The following manifolds have fixed points whose isotropy representations are isomorphic to the ones in $(\Sigma^7_k)$-$(\Sigma^{4n+1})$:
\begin{enumerate}
    \item $(\Sigma^7_k)$: any 3-sphere bundle over $\mathrm{S}^4$ ;
    \item $(\Sigma^8)$: every 3-sphere bundle over $\mathrm{S}^5$ or a 4-sphere bundle over $\mathrm{S}^4$;
    \item $(\Sigma^{10})$:
    \begin{enumerate}
        \item $M^8\times \mathrm{S}^2$ with $M^8$ as in item $(ii)$;
        \item any 3-sphere bundle over $\mathrm{S}^7$, 5-sphere bundle over $\mathrm{S}^5$ or 6-sphere bundle over $\mathrm{S}^4$;
    \end{enumerate}
    \item $(\Sigma^{4m+1},\mathrm{U}(n))$: \label{item:5} 
    \begin{enumerate}
        \item a sphere bundle $\mathrm{S}^{2m}\hookrightarrow M^{4m+1}\to \mathrm{S}^{2m+1}$ associated to any multiple
        of $\mathrm{O}(2m{+}1)\hookrightarrow \mathrm{O}(2m{+}2)\to \mathrm{S}^{2m+1}$, the frame bundle of $\mathrm{S}^{2m+1}$
        \item a $\mathbb  C\rm P^m$-bundle $\mathbb  C \rm P^{m}\hookrightarrow M^{4m+1}\to \mathrm{S}^{2m+1}$ associated to any multiple of the bundle of unitary frames $\mathrm{U}(m)\hookrightarrow \mathrm{U}(m+1)\to \mathrm{S}^{2m+1}$
        \item $M^{4m+1}=\frac{\mathrm{U}(m+2)}{S\mathrm{U}(2)\times \mathrm{U}(m)}$    
    \end{enumerate}
    \item $(\Sigma^{8r+5},\mathrm{Sp}(r))$: $M\times N^{5-k}$, where $N$ is any manifold and
    \begin{enumerate}
        \item $\mathrm{S}^{4r+k-1}\hookrightarrow M^{8r+k}\to \mathrm{S}^{4r+1}$ is the $k$-th suspension of the unitary tangent $\mathrm{S}^{4r-1}\hookrightarrow T_1\mathrm{S}^{4r{+}1}\to \mathrm{S}^{4r+1}$,
        \item $k=1$ and $\mathbb  H\rm P^{m}\hookrightarrow M^{8m+1}\to \mathrm{S}^{4m+1}$ is the $\mathbb  H\rm P^m$-bundle associated to any multiple of $\mathrm{Sp}(m)\hookrightarrow \mathrm{Sp}(m+1)\to \mathrm{S}^{4m+1}$
        \item $k=0$ and $M=\frac{\mathrm{Sp}(m+2)}{\mathrm{Sp}(2)\times \mathrm{Sp}(m)}$
        \item $k=1$ and $M=M^{8m+1}$ is as in item $(iv)$
    \end{enumerate}
\end{enumerate}    
\end{proposition}
\demo
\end{example}    

\section{The proof of Theorem \ref{thm:isospectralornot}}
\label{sec:basicspectra}
 In this last section, we prove Theorem \ref{thm:isospectralornot}. Let $M\leftarrow P \rightarrow M'$ denote a $\star$-diagram with structure group $G$ compact and connected. The manifolds $M, M'$, and $P$ are assumed to be closed and connected. Assume that $\mathsf{g}_M,~\mathsf{g}_{\omega},~\mathsf{g}_{M'}$ are as in Lemma \ref{lem:totallygeodesicmetric}. Proposition \ref{prop:joint} implies that the basic spectra of $-\Delta_{\mathsf{g}_M},~-\Delta_{\mathsf{g}_{M'}}$ coincide. Let $u$ be a $G$-invariant eigenfunction associated with $\lambda_1(\mathsf{g}_{M'})$ and consider the general vertical warping metric \cite[Chapter 2]{gw} $\mathsf{g}_u:=\pi^*(\mathsf{g}_M)+e^{2u}Q\circ \omega\otimes \omega$, where $Q$ is a bi-invariant metric in $G$ and $\omega$ is as in Lemma \ref{lem:totallygeodesicmetric}. Let $\mathsf{g}'$ be the Riemannian metric in $M'$ making $\pi':(P,\mathsf{g}_u)\rightarrow (M',\mathsf{g}')$ be a Riemannian submersion. It suffices to show that the first positive (invariant) eigenvalues $\lambda_1(\mathsf{g}_{M'}),~\lambda_1(\mathsf{g}')$ for $-\Delta_{\mathsf{g}_{M'}}$ and $-\Delta_{\mathsf{g}'}$, respectively, are different. 

 Take $\phi\in C^{\infty}_G(M')$ satisfying $-\Delta_{\mathsf{g}'}\phi=\lambda_1(\mathsf{g}')\phi$. Since $\phi$ lifts to a function on $P$ and $-\Delta_{\mathsf{g}'}\phi=-\Delta_{\mathsf{g}_u}\phi-\mathrm{d} \phi(H^{\pi'}_{\mathsf{g}_u})$ we have
 \[-\Delta_{\mathsf{g}_u}\phi-\mathrm{d}\phi(H^{\pi'}_{\mathsf{g}_u})=\lambda_1(\mathsf{g}')\phi.\]
 Therefore,
 \[-\int_P\mathrm{d}\phi(H^{\pi'}_{\mathsf{g}_u})\mathrm{d}\mu_{\mathsf{g}_u}=\lambda_1(\mathsf{g}')\int_P\phi\mathrm{d}\mu_{\mathsf{g}_u}.\]
 Equation 2.1.4 in \cite[Chapter 2, p.46]{gw} ensures that if $\mathrm{dim} G=k$ then $H^{\pi'}_{\mathsf{g}_u}=-k\nabla^{\mathsf{g}_u}u$. Consequently, 
\begin{align*}-\int_P\mathrm{d}\phi(H^{\pi'}_{\mathsf{g}_u})\mathrm{d}\mu_{\mathsf{g}_u}&=k\int_P\mathsf{g}_u(\nabla^{\mathsf{g}_u}\phi,\nabla^{\mathsf{g}_u}u)\mathrm{d}\mu_{\mathsf{g}_u}\\
&=k\int_P\mathsf{g}_{M'}(\nabla^{\mathsf{g}_{M'}}\phi,\nabla^{\mathsf{g}_{M'}}u)\mathrm{d}\mu_{\mathsf{g}_u}\\
&\leq k\int_P|\nabla^{\mathsf{g}_{M'}}u|_{\mathsf{g}_{M'}}|\nabla^{\mathsf{g}_{M'}}\phi|_{\mathsf{g}_{M'}}\mathrm{d}\mu_{\mathsf{g}_u}\\
&\leq k\left(\int_P|\nabla^{\mathsf{g}_{M'}}u|_{\mathsf{g}_{M'}}^2\mathrm{d}\mu_{\mathsf{g}_u}\right)^{\frac{1}{2}}\left(\int_P|\nabla^{\mathsf{g}_{M'}}\phi|^2_{\mathsf{g}_{M'}}\mathrm{d}\mu_{\mathsf{g}_u}\right)^{\frac{1}{2}},
\end{align*}
where the second equality holds because
	\[(TG\pi'(p))^{\perp_{\mathsf{g}'}}\cong(T(G\times G)p)^{\perp_{\mathsf{g}_{u}}} \cong (TG\pi(p))^{\perp_{\mathsf{g}_M}} \cong (TG\pi'(p))^{\perp_{\mathsf{g}_{M'}}},~\forall p\in P.\]

Using that $|\nabla^{\mathsf{g}_{M'}}u|_{\mathsf{g}_{M'}}^2$ is $G$-invariant and applying Fubini's Theorem for Riemannian submersions \cite[Satz 1, p.210]{fubinirefenrece} one gets
\begin{align*}
\int_{P}|\nabla^{\mathsf{g}_{M'}}u|_{\mathsf{g}_{M'}}^2\mathrm{d}\mu_{\mathsf{g}_u}&=\int_{M'}\left(\int_{G}|\nabla^{\mathsf{g}_{M'}}u|_{\mathsf{g}_{M'}}^2\mathrm{d}\mu_{Q}\right)\mathrm{d}\mu_{\mathsf{g}_{M'}}\\
&=\mathrm{vol}_u(G)\int_{M'}|\nabla^{\mathsf{g}_{M'}}u|_{\mathsf{g}_{M'}}^2\mathrm{d}\mu_{\mathsf{g}_{M'}}.
\end{align*}
The volume $\mathrm{vol}_u(G)$ is computed regarding each orbit's induced Riemannian metric by $\mathsf{g}_u$. Similarly,
\[\int_P|\nabla^{\mathsf{g}_{M'}}\phi|^2_{\mathsf{g}_{u}}\mathrm{d}\mu_{\mathsf{g}_{u}}=\mathrm{vol}_u(G)\int_{M'}|\nabla^{\mathsf{g}'}\phi|^2_{\mathsf{g}'}\mathrm{d}\mu_{\mathsf{g}'}.\]
Therefore, the variational characterization for the first eigenvalue (\cite[Section 5]{chavel}) yields
\begin{align*}
[\mathrm{vol}_u(G)]^{-1}k\left(\int_P|\nabla^{\mathsf{g}_{M'}}u|_{\mathsf{g}_{M'}}^2\mathrm{d}\mu_{\mathsf{g}_u}\right)^{\frac{1}{2}}\left(\int_P|\nabla^{\mathsf{g}_{M'}}\phi|^2_{\mathsf{g}_{M'}}\mathrm{d}\mu_{\mathsf{g}_u}\right)^{\frac{1}{2}}\geq k\\\lambda_1(\mathsf{g}_{M'})^{\frac{1}{2}}\lambda_1(\mathsf{g}')^{\frac{1}{2}}\left(\int_{M'}u^2\mathrm{d}\mu_{\mathsf{g}_{M'}}\right)^{\frac{1}{2}}\left(\int_{M'}\phi^2\mathrm{d}\mu_{\mathsf{g}'}\right)^{\frac{1}{2}}.
\end{align*}

Thus,
\[\lambda_1(\mathsf{g}')^{1/2}\left(\int_P\phi\mathrm{d}\mu_{\mathsf{g}_u}\right)\leq \lambda_1(\mathsf{g}_{M'})^{\frac{1}{2}}k\mathrm{vol}_u(G)\left(\int_{M'}u^2\mathrm{d}\mu_{\mathsf{g}_{M'}}\right)^{\frac{1}{2}}\left(\int_{M'}\phi^2\mathrm{d}\mu_{\mathsf{g}'}\right)^{\frac{1}{2}}.\]
Hence
\begin{align*}\left(\frac{\lambda_1(\mathsf{g}')}{\lambda_1(\mathsf{g}_{M'})}\right)^{\frac{1}{2}}&\leq k\mathrm{vol}_u(G)\left(\int_{M'}u^2\mathrm{d}\mu_{\mathsf{g}_{M'}}\right)^{\frac{1}{2}}\left(\int_P\phi\mathrm{d}\mu_{\mathsf{g}_u}\right)^{-1}\left(\int_{M'}\phi^2\mathrm{d}\mu_{\mathsf{g}'}\right)^{\frac{1}{2}}\\
&=k\left(\int_{M'}u^2\mathrm{d}\mu_{\mathsf{g}_{M'}}\right)^{\frac{1}{2}}\left(\int_{M'}\phi\mathrm{d}\mu_{\mathsf{g}'}\right)^{-1}\left(\int_{M'}\phi^2\mathrm{d}\mu_{\mathsf{g}'}\right)^{\frac{1}{2}}.\end{align*} 
Finally, one concludes the desired result as we can indiscriminately scale $u$ so that the right-hand side above is strictly lesser than $1$.
	
\begin{acknowledgement}
The authors sincerely thank the anonymous referees for their critical and detailed review, significantly improving the manuscript. They also express their gratitude to Emilio Lauret, whose numerous remarks helped them better acknowledge the existing literature and clarify the results of this paper. Part of this work was conceived while L.F.C was a postdoctoral researcher at the University of Fribourg, Switzerland, and he is grateful to Prof. Anand Dessai for his support and confidence during that time.
\end{acknowledgement}

\begin{funding}
L. F. Cavenaghi acknowledges The São Paulo Research Foundation (FAPESP), grants no 22/09603-9, 23/14316-1 and the SNSF-Project 200020E\_193062 and the DFG-Priority Program SPP 2026, which supported him while a postdoc at the University of Fribourg, where part of this work was conceived. 
 
 J. M. do \'O acknowledges partial support from CNPq through grants 312340/2021-4, 409764/2023-0, 443594/2023-6, CAPES MATH AMSUD grant 88887.878894/2023-00
and Para\'iba State Research Foundation (FAPESQ), grant no 3034/2021.   
\end{funding}

\begin{flushleft}
 {\bf Ethical Approval:}  Not applicable.\\
 {\bf Competing interests:}  Not applicable. \\
 {\bf Authors' contributions:}    All authors contributed to the study conception and design. All authors performed material preparation, data collection, and analysis. The authors read and approved the final manuscript.\\
{\bf Availability of data and material:}  Not applicable.\\
{\bf Ethical Approval:}  All data generated or analyzed during this study are included in this article.\\
{\bf Consent to participate:}  All authors consent to participate in this work.\\
{\bf Conflict of interest:} The authors declare no conflict of interest. \\
{\bf Consent for publication:}  All authors consent for publication. \\
\end{flushleft}


\begin{thebibliography}{10}

\bibitem{adelstein2017ginvariant}
I.~M. Adelstein and M.~R. Sandoval.
\newblock The $G$-invariant spectrum and non-orbifold singularities.
\newblock In {\em Archiv der Mathematik}. Springer, 109(6):563--573, 2017.

\bibitem{Alexandrino2022}
M.~M. Alexandrino and F.~Caramello.
\newblock Leaf closures of Riemannian foliations: a survey on topological and geometric aspects of Killing foliations.
\newblock \emph{Expositiones Mathematicae}, 40 (2022), 177-230.


\bibitem{alexandrino2016mean}
M.~Alexandrino and M.~Radeschi.
\newblock Mean curvature flow of singular Riemannian foliations.
\newblock {\em J. Geom. Anal.}, 26(3):2204--2220, 2016.

\bibitem{Blohmann2008}
C.~Blohmann.
\newblock Stacky Lie groups.
\newblock \emph{International Mathematics Research Notices}, 2008, article ID rnn082.

\bibitem{buser1982spectrum}
P.~Buser.
\newblock The spectrum of the laplacian of a {R}iemann surface and the Siegel conjecture.
\newblock {\em Commentarii Mathematici Helvetici}, 56(2):199--214, 1982.

\bibitem{cavenaghi2024newlookmilnor}
L.~F. Cavenaghi, L.~Gramma, and L.~Katzarkov.
\newblock New look at Milnor Spheres.
\newblock {\em arXiv eprint}, 2404.19088, 2024.

\bibitem{cavenaghi2019positive}
L.~Cavenaghi and L.~Sperança.
\newblock Positive Ricci curvature on fiber bundles with compact structure group.
\newblock {\em Adv. Geom.}, 22(1):95--104, 2022.

\bibitem{cavenaghi2023kazdanwarner}
L.~F. Cavenaghi, J.~M.~do \'O., and L.~D. Speran\c{c}a.
\newblock The Kazdan--Warner problem with (and via) symmetries.
\newblock {\em arXiv eprint}, 2106:14709, 2023.

\bibitem{Cavenaghi2023}
L.~F. Cavenaghi, R.~J. M., Silva, and L.~D. Sperança.
\newblock Positive Ricci curvature through Cheeger deformations.
\newblock Collect. Math, February 2023.

\bibitem{SperancaCavenaghiPublished}
L.~F. Cavenaghi and L.~D. Sperança.
\newblock On the geometry of some equivariantly related manifolds.
\newblock {\em Int. Math. Res. Not.}, page rny268, 2018.

\bibitem{chavel}
I.~Chavel, B.~Randol, and J.~Dodziuk.
\newblock {\em Pure and Applied Mathematics, Chapter I - The Laplacian}, volume 115.
\newblock Pages 1-25, 1984.
\newblock ISSN 0079-8169.

\bibitem{dryden2012advances}
E.~B. Dryden, V. Guillemin, and R. Sena-Dias.
\newblock Equivariant inverse spectral theory and toric orbifolds.
\newblock In {\em Advances in Mathematics}, Elsevier, Volume 231, Issues 3–4, pages 1271--1290, 2012.

\bibitem{duran2001pointed}
C.~Dur{\'a}n.
\newblock Pointed wiedersehen metrics on exotic spheres and diffeomorphisms of $\mathrm{S}^6$.
\newblock {\em Geom. Dedicata}, 88(1-3):199--210, 2001.

\bibitem{gordon1992one}
C.~Gordon, D.~L. Webb, and S.~Wolpert.
\newblock One cannot hear the shape of a drum.
\newblock {\em Bulletin of the American Mathematical Society}, 27(1):134--138, 1992.

\bibitem{gordon1996cant}
C.~Gordon and L.~D. Webb.
\newblock You can't hear the shape of a drum.
\newblock In {\em Proceedings of the International Congress of Mathematicians, Vol. 1}, pages 413--443, 1972.

\bibitem{gm}
D.~Gromoll and W.~Meyer.
\newblock An exotic sphere with nonnegative curvature.
\newblock {\em Ann. Math., 96}, 1972.

\bibitem{gw}
D.~Gromoll and G.~Walshap.
\newblock {\em Metric Foliations and Curvature}.
\newblock Birkhäuser, Basel, 2009.

\bibitem{gz}
K.~Grove and W.~Ziller.
\newblock Curvature and symmetry of milnor spheres.
\newblock {\em Ann. Math.}, 152:331--367, 2000.

\bibitem{Haefliger1984}
A.~Haefliger.
\newblock Groupoïdes d’holonomie et classifiants.
\newblock \emph{Astérisque}, 116 (1984), 70–97.

\bibitem{hebey1996sobolev}
E.~Hebey.
\newblock {\em Sobolev Spaces on Riemannian Manifolds}.
\newblock Springer, Number N{\textordmasculine} 1635 in Lect. Notes Math, 1996.

\bibitem{Hilsum1987}
M.~Hilsum and G.~Skandalis.
\newblock Morphismes K-orientés d’espaces de feuilles et functorialité en théorie de Kasparov.
\newblock \emph{Annales Scientifiques de l'École Normale Supérieure}, 20 (1987), 325-390.

\bibitem{hitchin1974harmonic}
N.~Hitchin.
\newblock Harmonic spinors.
\newblock {\em Adv. Math.}, 14(1):1--55, 1974.

\bibitem{shapeofadrum}
M.~Kac.
\newblock Can one hear the shape of a drum?
\newblock {\em The American Mathematical Monthly}, 73(4):1--23, 1966.

\bibitem{kazdaninventiones}
J.~L. Kazdan and F.~W. Warner.
\newblock A direct approach to the determination of {G}aussian and scalar curvature functions.
\newblock {\em Invent. Math.}, 28:227--230, 1975.

\bibitem{lauret2023spectrally}
E.~A.~Lauret and J.~S.~Rodríguez.
\newblock Spectrally distinguishing symmetric spaces I.
\newblock {\em arXiv eprint}, 2311.09719, 2023. 

\bibitem{lawson-yau}
H.~Lawson and S.-T. Yau.
\newblock Scalar curvature, non-abelian group actions, and the degree of symmetry of exotic spheres.
\newblock {\em Comm. Math. Helv.}, 49:232--244, 1974.

\bibitem{Lin2020}
Y.~Lin and R.~Sjamaar.
\newblock Cohomological localization for transverse Lie algebra actions on Riemannian foliations.
\newblock \emph{Journal of Geometry and Physics}, 158 (2020), article ID 103887.

\bibitem{mi}
J.~Milnor.
\newblock On manifolds homeomorphic to the $7$-sphere.
\newblock {\em Ann. Math.}, 64:399--405, 1956.

\bibitem{milnor1964eigenvalueII}
J.~Milnor.
\newblock Eigenvalue problems on symmetric Riemannian manifolds.
\newblock {\em Proceedings of the National Academy of Sciences of the United States of America}, 51(4):1054--1056, 1964.

\bibitem{milnor1964eigenvalues}
J.~Milnor.
\newblock Eigenvalues of the Laplace operator on certain manifolds.
\newblock {\em Proceedings of the National Academy of Sciences of the United States of America}, 51(4):542--544, 1964.

\bibitem{milnor1974characteristic}
J.~Milnor and J.~Stasheff.
\newblock {\em Characteristic Classes}.
\newblock Princeton Univ. Press, Ann. Math. Studies, 1974.

\bibitem{palais1979}
R.~Palais.
\newblock The principle of symmetric criticality.
\newblock {\em Commun. Math. Phys.}, 69(1):19--30, 1979.

\bibitem{Philippe1997}
T.~Philippe.
\newblock Geometry of foliations.
\newblock \emph{Springer Science \& Business Media}, 1997.


\bibitem{Richardson2001}
K.~Richardson.
\newblock The transverse geometry of G-manifolds and Riemannian foliations.
\newblock \emph{Illinois Journal of Mathematics}, 45.2 (2001), 517-535.


\bibitem{speranca2016pulling}
L.~Speran{\c{c}}a.
\newblock Pulling back the Gromoll--Meyer construction and models of exotic spheres.
\newblock {\em Proc. Amer. Math. Soc.}, 144(7):3181--3196, 2016.

\bibitem{fubinirefenrece}
R.~Sulanke and P.~Wintgen.
\newblock {\em Differentialgeometrie und Faserbundel}.
\newblock Math. Reihe Series, Birkhauser Boston, 1980.

\bibitem{sunada1985riemannian}
T.~Sunada.
\newblock Riemannian coverings and isospectral manifolds.
\newblock {\em Annals of Mathematics}, 121(1):169--186, 1985.

\bibitem{sutton2002isospectral}
C.~J. Sutton.
\newblock Isospectral simply-connected homogeneous spaces and the spectral rigidity of group actions.
\newblock In {\em Commentarii Mathematici Helvetici}. Springer, 77(4):701--717, 2002.

\bibitem{sutton2010equivariant}
C.~J. Sutton.
\newblock Equivariant isospectrality and sunada's method.
\newblock In {\em Archiv der Mathematik}. Springer, 95(1):75--85, 2010.

\bibitem{tanno1}
S.~Tanno.
\newblock Eigenvalues of the laplacian of riemannian manifolds.
\newblock {\em T\v{o}hoku Math. J.}, 25(3):391--403, 1973.

\bibitem{tanno2}
S.~Tanno.
\newblock The first eigenvalue of the laplacian on spheres.
\newblock {\em T\v{o}hoku Math. J.}, 31:2, 1979.

\bibitem{tanno3}
S.~Tanno.
\newblock Some metrics on a $(4r+3)$-sphere and spectra.
\newblock {\em Tsukuba J. Math.}, 4(1):99--105, 1980.

\bibitem{taubes}
C.~Taubes.
\newblock Gauge theory on asymptotically periodic $4$-manifolds.
\newblock {\em J. Differ. Geom.}, 25(3):363--430, 1987.

\bibitem{wiemeler2016circle}
M.~Wiemeler.
\newblock Circle actions and scalar curvature.
\newblock {\em Trans. Amer. Math. Soc.}, 368(4):2939--2966, 2016.

\end{thebibliography}
\end{document}